\documentclass[a4paper,12pt]{amsart}
\usepackage{amsmath}
\usepackage{amssymb}
\usepackage{amsfonts}
\usepackage{amscd}
\usepackage{mathrsfs}
\usepackage{paralist}
\usepackage{url}
\usepackage{enumerate}
\usepackage[all]{xy}
\usepackage[dvipdfmx]{graphicx}
\usepackage{graphics}
\usepackage{tikz}
\everymath{\displaystyle}
\theoremstyle{plain} 
\newtheorem{theorem}{\indent\bf Theorem}[section]
\newtheorem{lemma}[theorem]{\indent\bf Lemma}

\newtheorem{proposition}[theorem]{\indent\bf Proposition}
\newtheorem{claim}[theorem]{\indent\bf Claim}

\theoremstyle{definition} 

\newtheorem{remark}[theorem]{\indent\rm Remark}
\newtheorem{example}[theorem]{\indent\rm Example}

\newtheorem*{theorem*}{\indent\bf Theorem}

\newcommand{\del}{\partial}
\newcommand{\delbar}{\overline{\partial}}
\newcommand{\ddbar}{\del\delbar}

\def\Re{\text{Re\,}} 
\def\Im{\text{Im\,}} 
\newcommand{\C}{\mathbb{C}}

\newcommand{\R}{\mathbb{R}}
\newcommand{\N}{\mathbb{N}}
\newcommand{\Z}{\mathbb{Z}}

\renewcommand{\l}{\ell}
\renewcommand{\i}{\mathrm{i}}
\renewcommand{\H}{\mathbb{H}}
\newcommand{\hsp}{\hspace{2mm}}

\renewcommand{\hsp}{\hspace{2mm}}
\newcommand{\lam}{\lambda}
\newcommand{\tlam}{\tilde{\lambda}}

\newcommand{\eps}{\varepsilon}

\newcommand{\dist}{{\rm{dist}}}
\newcommand{\Supp}{{\rm Supp}\,}
\newcommand{\Ker}{{\rm Ker}}
\newcommand{\Image}{{\rm Im}}
\renewcommand{\bar}{\overline}

\renewcommand{\tilde}{\widetilde}
\renewcommand{\hat}{\widehat}
\renewcommand{\i}{\sqrt{-1}}

\begin{document}

\pagestyle{plain}
\thispagestyle{plain}

\title[Cohomology groups with compact support for flat line bundles on certain complex Lie groups]
{Cohomology groups with compact support for flat line bundles on certain complex Lie groups}

\author[Takayuki KOIKE and Jinichiro Tanaka]{Takayuki KOIKE$^{1}$ and Jinichiro Tanaka$^{2}$}
\address{ 
$^{1}$ Department of Mathematics \\
Graduate School of Science \\
Osaka Metropolitan University \\
3-3-138 Sugimoto \\
Osaka 558-8585 \\
Japan 
}
\email{tkoike@omu.ac.jp}
\address{
$^{2}$ Department of Mathematics \\
Graduate School of Science \\
Osaka Metropolitan University \\
3-3-138 Sugimoto \\
Osaka 558-8585 \\
Japan 
}
\email{sw23876x@st.omu.ac.jp}
\renewcommand{\thefootnote}{\fnsymbol{footnote}}
\renewcommand{\thefootnote}{\arabic{footnote}}
\begin{abstract} 
Let $X$ be a complex surface obtained as the quotient of the complex Euclidean space $\mathbb{C}^2$ by a discrete subgroup of rank $3$. 
We investigate the cohomology group $H_0^1(X, E)$ with compact support for a unitary flat line bundle $E$ over $X$. 
We show the vanishing of $H_0^1(X, E)$ for a certain class of such pairs $(X, E)$, 
which includes infinitely many examples such that $H^1(X, E)$ is non-Hausdorff and infinite dimensional.
\end{abstract}

\maketitle

\section{Introduction}\label{section:introduction}

A {\it toroidal group} is a complex manifold obtained as the quotient $\C^n/\Gamma$ of the complex Euclidean space $\C^n$ by a discrete subgroup $\Gamma\subset \C^n$ which has no non-constant holomorphic function. 
Toroidal groups are classified into {\it toroidal theta groups} and {\it toroidal wild groups} in accordance with irrational number theoretical conditions on $\Gamma$ (see \S2). 

In the following, we consider cohomology groups as Dolbeault cohomology groups.  
For a toroidal group $X$, the cohomology group $H^k(X, \mathcal{O}_X)$ is determined by Kazama \cite{Kazama}. 
The dimension of $H^k(X, \mathcal{O}_X)$ is a finite number which can be explicitly described when $X$ is a toroidal theta group, whereas $H^k(X, \mathcal{O}_X)$ is non-Hausdorff and infinite dimensional when $X$ is a toroidal wild group. 
As a variant of Kazama's theorem, Abe determined the cohomology group $H^k(X, E)$ for a unitary flat line bundle $E$ over a toroidal group $X$ \cite{Abetoro}. 

In the present paper, we investigate the cohomology group $H_0^1(X,E)$ with compact support for $X$ and $E$ as in the following.
\begin{example}\label{eg:main}
For $\tau\in\H:=\{z\in \mathbb{C}\mid {\rm Im}\,z>0\}$ and real numbers $p$ and $q$, set $X:= X_{\tau,p,q}:= \mathbb{C}^2/\Gamma$, where $\Gamma := \Gamma_{\tau,p,q} \subset \mathbb{C}^2$ is the discrete group defined by 
\[
\Gamma_{\tau,p,q}:= \left\langle 
\begin{pmatrix}
0 \\
1 \\
\end{pmatrix},\ 
\begin{pmatrix}
1 \\
p \\
\end{pmatrix},\ 
\begin{pmatrix}
\tau \\
q \\
\end{pmatrix}
\right\rangle.
\]
For $\theta_1$ and $\theta_2 \in \R$, denote by $E:=E_{\theta_1,\theta_2}$ the unitary flat line bundle over $X$ which corresponds to the unitary representation $\rho$ of the fundamental group $\pi_1(X,\ast)= \Gamma$ defined by 
\[
\rho\left(\begin{pmatrix}
0 \\
1 \\
\end{pmatrix}\right) :=1 ,\quad
\rho\left(\begin{pmatrix}
1 \\
p \\
\end{pmatrix}\right) :=e^{2\pi\i\theta_1} ,\quad{\rm and}\ 
\rho\left(\begin{pmatrix}
\tau \\
q \\
\end{pmatrix}\right) :=e^{2\pi\i\theta_2}. 
\]
\end{example}

Note that the complex surface $X$ in Example \ref{eg:main} is not necessarily a toroidal group (see Lemma \ref{toroidaliff}). 

For Example \ref{eg:main}, we show the following:
\begin{theorem}\label{thm:main}
Let $\tau$ be an element of $\H$ and $p, q, \theta_1$,and $\theta_2$ real numbers.
Assume either of the following for any $n\in\Z \colon$ 
$\theta_1+np\notin\Z$ or $\theta_2+nq\notin\Z$. 
Then $H_0^1(X, E)=0$ holds for $X=X_{\tau,p,q}$ and $E=E_{\theta_1,\theta_2}$ as in Example \ref{eg:main}. \qed
\end{theorem}

Note that $H^0(X,E)$ vanishes under the assumption in Theorem \ref{thm:main} for $X=X_{\tau,p,q}$ and $E=E_{\theta_1,\theta_2}$ (see Lemma \ref{H0XE}). 
A pair $(X,E)$ as in Example \ref{eg:main} includes many examples such that $H^1(X,E)$ is non-Hausdorff and infinite dimensional. 
For example, when $X=X_{\tau,0,q}$ is a toroidal group and $E=E_{0,\theta_2}$ is a non-trivial unitary flat bundle, the cohomology group $H^1(X, E)$ becomes either $0$ or, non-Hausdorff and infinite dimensional under the assumption in Theorem \ref{thm:main} (see Remark \ref{rmk:general}). 

We sketch the proof of Theorem \ref{thm:main}. 
Denote by $\mathscr{A}_{C^\infty}^{p,q}(E)$ the space of $E$-valued $(p,q)$-forms of class $C^\infty$ on $X$. 
We consider $\mathscr{A}_{C^\infty}^{p,q}(E)$ with the topology of the uniform convergence of all the derivatives of coefficients on compact sets.
We show that the space $B^{0,1}(E)$ of $(0,1)$-coboundaries is dense in the space $Z^{0,1}(E)$ of $(0,1)$-cocycles via the \v{C}ech--Dolbeault correspondence, from which the vanishing of the topological dual space $(H^1(X,E))^\vee$ follows. 
The Serre duality theorem implies that $H_0^1(X,E)$ is isomorphic to $(H^1(X,E^{-1}))^\vee$. 

In Example \ref{eg:main} with $p=\theta_1=0$, we give another proof of Theorem \ref{thm:main} by directly constructing a solution of the $\delbar$-equation. 
The approach is as follows: 
We take $[f] \in H^1_0(X,E)$ which is represented by an $E$-valued $C^\infty$ $(0, 1)$-form $f$ on $X$ with compact support. 
As is observed in \S \ref{obs}, there is a covering map $p\colon (\C^\ast)^2 \to X$, where $\C^\ast := \C\setminus \{0\}$. 
Define a $(0,1)$-form $\widetilde{f}$ on $(\C^\ast)^2$ by letting $\widetilde{f}:=p^*f$. 
For constructing a smooth section $g$ of $E$ with compact support such that $\delbar g = f$ holds, 
we will solve the $\delbar$-equation $\delbar \widetilde{g} = \widetilde{f}$ on $(\C^\ast)^2$ so that the solution $\widetilde{g}$ enjoys a periodic condition. 
For this purpose, we will construct a suitable K\"ahler metric and a weight function on $(\C^\ast)^2$ and apply H\"ormander's theorem on the existence of a solution of a $\delbar$-equation with an $L^2$ estimate. 


\vskip3mm
{ Organization of the paper. }
In \S \ref{toroidal}, we review some known results for $X$ and $E$ as in Example \ref{eg:main}. 
In \S \ref{general:proof}, we prove Theorem \ref{thm:main}. 
In \S \ref{section:proof}, we give another proof of Theorem \ref{thm:main} in the restricted case of Example \ref{eg:main}.

\vskip3mm
{ Acknowledgement. }
The authors would like to give heartful thanks to Prof. Takeo Ohsawa and Prof. Shigeharu Takayama for helpful comments. 
The first author is supported by JSPS KAKENHI Grant Number 23K03119. 
The second author is supported by JST, the establishment of university fellowships towards the creation of science technology innovation, Grant Number JPMJFS2138. 
This work was supported by MEXT Promotion of Distinctive Joint Research Center Program JPMXP0723833165, and the Research Institute for Mathematical Sciences, an International Joint Usage/Research Center located in Kyoto University.

\section{Some fundamentals}\label{toroidal}

In this section, we explain some known results for $X=X_{\tau,p,q}$ and $E=E_{\theta_1,\theta_2}$ as in Example \ref{eg:main}. 

\begin{lemma}\label{H0XE}
Let $X=X_{\tau,p,q}$ and $E=E_{\theta_1,\theta_2}$ be as in Example $\ref{eg:main}$. 
Then $H^0(X,E)$ vanishes under the assumption in Theorem \ref{thm:main}. 
\end{lemma}

\begin{proof}
Denote by $C$ the elliptic curve $\C/\left\langle 1, \tau \right\rangle$. 
Let $\pi \colon X \to C$ be the mapping induced by the first projection $\C^2 \to \C \ ; (z,w)\mapsto z$. 
Note that $\pi \colon X\to C$ can be regarded as a $\C^\ast$-fiber bundle structure on $C$. 
Let $\{U_j\}_j$ be an open covering of $C$. 
Take a local coordinate $z_j$ on $U_j$. 
Let $F_C$ be the flat line bundle over $C$ which corresponds to the unitary representation $\rho_{F_C}$ of $\pi_1(C,\ast) =\left\langle 1, \tau \right\rangle$ defined by letting $\rho_{F_C} (1) = e^{2\pi\i p}$ and $\rho_{F_C} (\tau) = e^{2\pi\i q}$.
Then $X$ can be identified with the complement of the zero section in the total space of $F_C$.
Take fiber coordinates $w_j$ of $F_C|_{U_j}$. 
There exist holomorphic functions $s_{jk}$ on $U_{jk}:=U_j \cap U_k$ such that $w_j = s_{jk}^{-1} w_k$ on $U_{jk}$. 
Let $E_C$ be the flat line bundle over $C$ which corresponds to the unitary representation $\rho_{E_C}$ defined by letting $\rho_{E_C} (1) = e^{2\pi\i\theta_1} $ and $\rho_{E_C} (\tau) = e^{2\pi\i\theta_2}$. 
Note that $E$ can be regarded as the flat line bundle $\pi^\ast E_C$ on $X$. 
Take local trivializations $e_j$ of ${E_C}|_{U_j}$. 
There exist holomorphic functions $t_{jk}$ on $U_{jk}$ such that $e_j= t_{jk}^{-1} e_k$ on $U_{jk}$. 
Take $f\in H^0(X,E)$. 
We can expand $f$ as 
\begin{align*}
f(z_j,w_j) = \left(\sum_{n=-\infty}^\infty a_n^j (z_j) w_j^n \right)\pi^\ast e_j
\end{align*}
on $U_j$ where each $a_n^j$ is a holomorphic function on $U_j$. 
As 
\begin{align*}
f(z_j, w_j) =  \left(\sum_{n=-\infty}^\infty a_n^j (z_j) s_{jk}^{-n} t_{jk}^{-1} w_k^n \right)\pi^\ast e_k
\end{align*} 
on $U_{jk}$, the equation
\begin{align*}
a_n^k(z_k) = s_{jk}^{-n} t_{jk}^{-1} a_n^j (z_j)
\end{align*}
holds on $U_{jk}$. 
Thus, we find that $\{ ( U_j , a_n^j ) \}\in H^0(C, {F_C}^n\otimes {E_C})$ for each $n\in\Z$. 
By the assumption in Theorem \ref{thm:main}, ${F_C}^n\otimes {E_C}$ is a non-trivial flat line bundle. 
We have $H^0(C, {F_C}^n\otimes {E_C})=0$ for each $n\in\Z$ (see \cite[Lemma 2.3]{HaKo} for example). 
Therefore $f\equiv 0$.
\end{proof}

\begin{lemma}\label{toroidaliff}
The complex surface $X_{\tau,p,q}$ is a toroidal group if and only if either $p$ or $q$ is an irrational number. 
\end{lemma}

\begin{proof}
The assertion is a direct consequence of \cite[Theorem $1.1.4$]{AK}. 
\end{proof}

In the following, we assume that $X$ is a toroidal group. 
It is known that the following condition is important to classify $X_{\tau,p,q}$ (see \cite{KaTa} or \cite{Kodel}): 
There exist positive numbers $A$ and $0<\delta<1$ such that, for any positive interger $n$, 
\begin{align*}
\dist ((np,nq), \Z^2)\geq A\delta^n
\end{align*}
holds, where $\dist$ is the Euclidian distance.
If this condition holds, $X$ is said to be a {\it toroidal theta group}. 
Otherwise $X$ is said to be a {\it toroidal wild group}.

We explain some consequences from some known results such as the theorems due to Kazama \cite{Kazama} and Abe \cite{Abetoro} through Example \ref{eg:main}. 
The following is known for the cohomology group $H^k(X,\mathcal{O}_X)$\,\cite{Kazama}. 

\begin{proposition}{\rm (\cite[Theorem $4.3$]{Kazama} for Example \ref{eg:main})}\label{Kazama:dimthm}
The following assertions are equivalent when $X=X_{\tau,p,q}$ as in Example \ref{eg:main} is a toroidal group. 
\begin{enumerate}[$(i)$]
\item $X$ is a toroidal theta group.
\item $\dim H^1(X, \mathcal{O}_{X})=1. $
\item $\dim H^k (X, \mathcal{O}_X) <\infty $ for any $k\geq1$.
\end{enumerate}
Moreover, the space $H^k(X, \mathcal{O}_X)$ is non-Hausdorff and infinite dimensional when $X$ is a toroidal wild group.
\end{proposition}

\begin{proof}
After a linear change of the coordinates, we obtain
\begin{align*}
X=\C^2/\left\langle \begin{pmatrix} 0\\ 1\\ \end{pmatrix} , \begin{pmatrix} 1\\ p\\ \end{pmatrix} , \begin{pmatrix} \tau\\ q\\ \end{pmatrix} \right\rangle
\cong \C^2/\left\langle \begin{pmatrix} 1\\ 0\\ \end{pmatrix} , \begin{pmatrix} 0\\ 1\\ \end{pmatrix} , \begin{pmatrix} s\\ t\\ \end{pmatrix} \right\rangle
\end{align*}
where $s:=\tau$ and $t:=q-p\tau$. 
By \cite[Theorem 4.3]{Kazama}, it suffices to show the equivalence of $(i)$ and the following condition: 
\begin{enumerate}[$(iv)$]
\item For $m=(m_1,m_2,m_3)\in\Z^3$, there exists $a>0$ such that
\begin{align*}
\sup_{m\in\Z^3\setminus\{0\}} \frac{e^{-a \cdot\max\{|m_1|,|m_2|\}}}{|s m_1+t m_2 - m_3|}<\infty.
\end{align*}
\end{enumerate}
(We apply \cite[Theorem 4.3]{Kazama} with $v_{11}=s$, $v_{12}=t$, and $K_m=K_{m,1}=|s m_1+t m_2-m_3|$ for $m\in\Z^3$, where $v_{11}$, $v_{12}$, and $K_m$ are as in \cite{Kazama}. See also \cite{Kodel}.) 
Note that the inequality in the assertion $(iv)$ is equivalent to 
\begin{align}
\sup_{m\in\Z^3\setminus\{0\}} \frac{e^{-a \cdot\max\{|m_1|,|m_2|\}}}{|\tau(pm_2- m_1)+ m_3 -q m_2|} <\infty. \label{dimcondi}
\end{align}

Assume that there exists $a>0$ such that $(\ref{dimcondi})$ holds. 
By the equivalence of norms in a finite dimensional vector space, we can take $k_{\tau,p,q}>0$ such that
\begin{align*}
|\tau(pm_2- m_1)+ m_3 -q m_2|< k_{\tau,p,q} \cdot\dist((m_2 p, m_2 q),(m_1,m_3))
\end{align*}
for all $m=(m_1,m_2,m_3) \in \Z^3\setminus\{0\}$. 
Take $M>0$ such that
\begin{align*}
\frac{e^{-a \cdot\max\{|m_1|,|m_2|\}}}{|\tau(pm_2- m_1)+ m_3 -q m_2|}<M
\end{align*} 
for all $m\in\Z^3\setminus\{0\}$. 
Then we have
\begin{align*}
\frac{e^{-a \cdot\max\{|m_1|,|m_2|\}}}{k_{\tau,p,q}M}< \dist((m_2 p, m_2 q),(m_1,m_3))
\end{align*}
for all $m\in\Z^3\setminus\{0\}$. 
Now, we can assume that  $|m_2p-m_1|<1$ holds. 
Hence 
\begin{align*}
\frac{e^{-a (|pm_2|+1+|m_2|)}}{k_{\tau,p,q}M} \leq \frac{e^{-a (|m_1|+|m_2|)}}{k_{\tau,p,q}M}\leq \dist((m_2 p, m_2 q),\Z^2)
\end{align*}
holds for all $m_3\in\N$. Set $\delta:=e^{-a(|p|+1)}$ and $ A:=\frac{e^{-a}}{k_{\tau,p,q}M}$. 
We find that the condition $(i)$ holds.

On the other hand, assume that the condition $(i)$ holds. 
Note that $A\delta^n \leq\dist ((n p, n q),(m_1,m_2))$ holds for all $(m_1, m_2)\in\Z^2$ and there exists $\tilde{k}_{\tau,p,q}>0$ such that
\begin{align}
|\tau(np- m_1)+ m_2 -nq|> \tilde{k}_{\tau,p,q} \cdot\dist((n p, n q),(m_1,m_2)). \label{ktilcondi}
\end{align}
We find that
\begin{align*}
\frac{\delta^n}{|\tau(np- m_1)+ m_2-nq|}<\frac{1}{\tilde{k}_{\tau,p,q}A}
\end{align*}
for all $(m_1,m_2)\in\Z^2$ and any $n\in\Z_{>0}$. 
Then 
\begin{align*}
\frac{\delta^{\max\{|m_1|,|n|\}}}{|\tau(np- m_1)+ m_2-nq|}<\frac{1}{\tilde{k}_{\tau,p,q}A}
\end{align*}
holds. 
By setting $a:=-\log\delta$, we obtain the inequality $(\ref{dimcondi})$. 
Therefore, we find that the condition $(iv)$ holds.
\end{proof}

By Abe's results on $H^k(X, E)$ for a holomorphic line bundle $E$ over a general toroidal group $X$ (see \cite{Abetoro}), we can specify the dimension of $H^k(X, E)$ as follows. 

\begin{proposition}{\rm (\cite[Theorem $9.1$]{Abetoro} for Example \ref{eg:main} with $p=\theta_1=0$)}\label{Abe:finitethm}
Let $X=X_{\tau, 0, q}$ and $E=E_{0, \theta_2}$.
Assume that $X$ is a toroidal group and $E$ is not holomorphically trivial. 
Then $H^1(X,E)=0$ holds if the following condition holds: there exist $A>0$ and $0<\delta<1$ such that
\begin{align}\label{Abe:condi}
\dist(nq-\theta_2,\Z)\geq A\delta^n
\end{align}
for all $n\in\Z_{>0}$.
If this condition does not hold, $H^1(X,E)$ is non-Hausdorff and infinite dimensional. \qed
\end{proposition}

Note that, according to Proposition \ref{Abe:finitethm}, $H^1(X,E)$ vanishes if $\tau=\i$, $\theta_2=q$, and $q$ is a {\it Diophantine number}, that is, there exist $A>0$ and $N\in\Z_{>0}$ such that 
\begin{align*}
\left| q-\frac{m}{n} \right|\geq \frac{A}{n^N}
\end{align*}
for any integers $m$ and $n$ with $n>0$.
If $q$ is a {\it Liouville number}, it does not hold (See \cite[Example $10.2$ and $10.3$]{Abetoro}).

\begin{proof}[\rm Proof of Theorem \ref{Abe:finitethm}]
By \cite[Theorem $9.1$]{Abetoro}, it suffices to prove that the condition in Theorem \ref{Abe:finitethm} is equivalent to the following condition $(H)'_S$: there exist $C>0$ and $a>0$ such that
\begin{align*}
| \tau m_1+q m_2-m_3-\theta_2 |\geq C e^{-a\max\{|m_1|,|m_2|\}}
\end{align*}
for all $m=(m_1,m_2,m_3)\in\Z^3$ with $(m_1,m_2)\neq(0,0)$. 

If $(H)'_S$ holds, 
\begin{align*}
| q m_2-m_3-\theta_2 |\geq C e^{-a|m_2|}
\end{align*}
holds for any $m_2\in\Z\setminus\{0\}$. We set $A:=C$ and $\delta:=e^{-a}$, and then we obtain the inequality $(\ref{Abe:condi})$. 
On the other hand, assume that there exist $A>0$ and $0<\delta<1$ such that the inequality $(\ref{Abe:condi})$ holds for all $n\in\Z_{>0}$. 
By the equivalence of norms in a finite dimensional vector space, there exists $k_{\tau,q,\theta_2}>0$ such that
\begin{align*}
|\tau m_1+nq -m_2 -\theta_2|\geq k_{\tau,q,\theta_2} |nq -\theta_2 -m_2|
\end{align*}
for all $n\in\Z_{>0}$ and $(m_1,m_2)\in\Z^2$.  Hence,
\begin{align*}
|\tau m_1+nq -m_2 -\theta_2|&\geq k_{\tau,q,\theta_2}\dist(nq-\theta_2,\Z) \\
&\geq Ak_{\tau,q,\theta_2}\delta^n \geq Ak_{\tau,q,\theta_2} \delta^{\max\{n,|m_1|\}}
\end{align*}
holds. Set $C:=Ak_{\tau,q,\theta_2}$ and $a:=-\log\delta$. Thus, we find that $(H)'_S$ holds.
\end{proof}

\begin{remark}\label{rmk:general}
For Example \ref{eg:main}, $H^1(X,E)$ is 0, or non-Hausdorff and infinite dimensional under the assumption in Theorem \ref{thm:main} when $X$ is a toroidal gruoup. In fact, we know the following by \cite[$(7.3)$ and Theorem $9.1$]{Abetoro}; there does not exist $\sigma=(\sigma_1,\sigma_2,\sigma_3)\in\Z^3\setminus\{0\}$ satisfying
\begin{align*}
 \left\{ \begin{array}{ll}
 \Re\tau (\sigma_1-p\sigma_2+\theta_1)+\sigma_2q+\sigma_3-\theta_2 =0 \\
 \Im\tau (\sigma_1-p\sigma_2+\theta_1)=0
\end{array} \right.
\end{align*}
since $\theta_1+np\notin\Z$ or $\theta_2+nq\notin\Z$ for any integer $n$. 
Hence, we find that $Z$ in \cite{Abetoro} is $\Z^3$. Thus, $H^1(X,E)$ is 0, or non-Hausdorff and infinite dimensional. \qed
\end{remark}

\section{Proof of Theorem \ref{thm:main}}\label{general:proof}

\subsection{The density of $B^{0,1}(E)$ in $Z^{0,1}(E)$}\label{notation:general}
Here we show the vanishing of $(H^1(X,E))^\vee$ in a slightly generalized configuration. 
Let $(X,E)$ be a pair of a complex manifold $X$ and a holomorphic line bundle $E$ over $X$ which satisfies the following assumption. 
There exist a compact complex manifold $Y$ and a holomorphic line bundle $\pi \colon F_Y\to Y$ such that the following conditions hold: 
$(i)$ The complement of the zero section in the total space of $F_Y$ is biholomorphic to $X$. 
$(ii)$ There exists a holomorphic line bundle $E_Y$ over $Y$ such that $\pi^\ast E_Y =E$. 
$(iii)$ For any $n\in\Z$, $H^1(Y, F_Y^n\otimes E_Y)=0$. 

\begin{lemma}
Let $(X,E)$ be as in Example \ref{eg:main}. The pair $(X,E)$ satisfies the conditions $(i)$, $(ii)$, and $(iii)$ under the assumption in Theorem \ref{thm:main}.
\end{lemma}

\begin{proof}
Let $Y$ be the elliptic curve $\C/\left\langle 1, \tau \right\rangle$. 
We define $F_Y:=F_C$ and $E_Y:=E_C$ as in the proof of Lemma $\ref{H0XE}$. 
As already observed in the proof of Lemma $\ref{H0XE}$, we easily find that the conditions $(i)$ and $(ii)$ hold. 
We find that $F_Y^n\otimes E_Y$ is not trivial by the assumption in Theorem \ref{thm:main}. 
Hence, the condition $(iii)$ holds (see \cite[Lemma 2.3]{HaKo} for example).
\end{proof}

Let $\mathscr{A}^{p,q}_{C^\infty}(E)$ be the set of $E$-valued $(p,q)$-forms of class $C^\infty$ on $X$. Set 
\begin{align*}
Z^{p,q}(E)&:=\Ker \left( \delbar \colon \mathscr{A}^{p,q}_{C^\infty}(E) \to \mathscr{A}^{p,q+1}_{C^\infty}(E) \right) \ \text{and}\\
B^{p,q}(E)&:= \Image \left( \delbar \colon \mathscr{A}^{p,q-1}_{C^\infty}(E) \to \mathscr{A}^{p,q}_{C^\infty}(E) \right).
\end{align*}

\begin{proposition}\label{prop:dense}
Let $(X,E)$ be a pair which satisfies the conditions $(i)$, $(ii)$, and $(iii)$ above. 
In the topology of the uniform convergence of all the derivatives of coefficients on compact sets, $B^{0,1}(E)$ is dense in $Z^{0,1}(E)$.
\end{proposition}

\begin{proof}
Let $\{ U_\lam \}_\lam$ be a sufficiently fine finite Stein open covering of $Y$. 
Take a local coordinate $\xi_\lam$ on $U_\lam$. 
The family $\{ V_\lam \}$ defined by $V_\lam:=\pi^{-1}(U_\lam)\, \cap\,  X$ is an open covering of $X$. 
We denote $\pi^\ast\xi_\lam$ by $\xi_\lam$ for simplicity. 
Denote by $\eta_\lam\in\C^\ast$ the fiber coordinate of $V_\lam$. 
We regard $(\xi_\lam,\eta_\lam)$ as coordinates of $V_\lam$. 
For $(\xi_j,\eta_j)=(\xi_k,\eta_k)$ on $V_{jk}:=V_j \cap V_k$, there exist holomorphic functions $t_{jk} \colon U_{jk}:=U_j \cap U_k \to \C^\ast$ such that $\xi_j=\xi_k$ on $U_{jk}$ and $\eta_j= t_{jk}^{-1}\eta_k$. 
Take a local trivialization $e_j$ of $E_Y|_{U_j}$ and let $s_{jk}\colon U_{jk} \to \C^\ast$ be the holomorphic function such that $e_j = s_{jk}^{-1}e_k$.

Take $w\in Z^{0,1}(E)$. Let $[ \{ (V_{jk}, f_{jk} ) \} ]\in \check{H}^1(\{V_j \}, \mathcal{O}_X(E))$ be the element which coincides with $[ w ] \in H^{0,1}(X,E)$ via the \v{C}ech--Dolbeault correspondence. 
Denote $\pi^\ast e_j$ by $e_j$. 
We can expand each $f_{jk}$ as
\begin{align*}
f_{jk}=\left( \sum_{n=-\infty}^\infty a_{jk,n}(\xi_j)\eta_j^n \right) e_j.
\end{align*}
By the $1$-cocycle conditions for $f_{jk}$'s, $f_{jk}+f_{k\l}+f_{\l j}\equiv 0$ on $V_{jk\l}:=V_j\cap V_k \cap V_\l$. 
We have
\begin{align*}
0 &= \sum_{n=-\infty}^\infty a_{jk,n}\eta_j^n e_j+ \sum_{n=-\infty}^\infty a_{k\l,n}\eta_k^n e_k+ \sum_{n=-\infty}^\infty a_{\l j,n}\eta_\l^n e_\l\\
&=\sum_{n=-\infty}^\infty a_{jk,n}\eta_j^n e_j+ \sum_{n=-\infty}^\infty a_{k\l,n}(t^{-1}_{kj}\eta_j)^n (s_{kj}^{-1} e_j)+ \sum_{n=-\infty}^\infty a_{\l j,n}(t^{-1}_{\l j}\eta_j)^n (s_{\l j}^{-1} e_j)\\
&=\sum_{n=-\infty}^\infty \left( a_{jk,n}+ a_{k\l,n} t^{-n}_{kj} s_{kj}^{-1} + a_{\l j,n} t^{-n}_{\l j}s_{\l j}^{-1} \right) \eta_j^n e_j.
\end{align*}
Thus, we obtain the relation
\begin{align}\label{eq:cocycle1}
a_{jk,n}+ a_{k\l,n} t^{-n}_{kj} s_{kj}^{-1} + a_{\l j,n} t^{-n}_{\l j}s_{\l j}^{-1} \equiv 0
\end{align}
on $U_{jk\l}$ for $n\in\Z$. 
Let $\eps_j$ be local trivializations of $F_Y|_{U_j}$ such that $\eps_j=t_{kj}\eps_k$. 
By multiplying the equation $(\ref{eq:cocycle1})$ by $\eps_j^n\otimes e_j$, 
\begin{align*}
a_{jk,n} \eps_j^n \otimes e_j + a_{k\l,n} \eps_k^n\otimes e_k + a_{\l j,n} \eps_\l^n \otimes e_\l \equiv 0,
\end{align*}
from which we find that $\{(U_{jk}, a_{jk,n}\eps_j^n \otimes e_j)\}\in Z^{0,1}(F_Y^n\otimes E_Y)$. 
There exists $\{(U_j, b_{j,n}\eps_j^n \otimes e_j)\}\in \check{C}^0( \{ U_j \}, \mathcal{O}_Y(F_Y^n\otimes E_Y))$ such that $\delta (\{(U_j, b_{j,n}\eps_j^n \otimes e_j)\} )= \{(U_{jk},a_{jk,n}\eps_j^n \otimes e_j)\}$ since $H^1(Y, F_Y^n\otimes E_Y)=0$ for all $n\in\Z$. 
Set the formal power series $g_j$ and the partial sum of $g_j^{(N)}$ for $N\in\Z_{>0}$ as follows:
\begin{align*}
g_j&:=\left(\sum_{n=-\infty}^\infty b_{j,n}\eta_j^n\right) e_j \ \text{and}\\
g_j^{(N)}&:=\left(\sum_{n=-N}^N b_{j,n}\eta_j^n\right) e_j.
\end{align*}
Similarly, define the partial sum of $f_{jk}$ by $\textstyle f_{jk}^{(N)}:=\left( \sum_{n=-N}^N a_{jk,n}(\xi_j)\eta_j^n \right) e_j$. 
Note that $g_k-g_j=f_{jk}$ holds formally, and that $f_{jk}^{(N)}$ converges to $f_{jk}$ on $V_{jk}$ as $N\to \infty$. 
By constructions of $g_j^{(N)}$ and $f_{jk}^{(N)}$, we have $\delta (\{(V_j, g_j^{(N)}) \})=\{ (V_{jk}, f_{jk}^{(N)} )\}$. 
Let $\{\rho_j\}$ be a partition of unity which is subordinate to the finite open cover $\{U_j\}$. 
Denote $\pi^\ast\rho_j$ by $\rho_j$. 
For $g_j^{(N)}$, we set $G^{(N)}:=\sum_j g_j^{(N)}\rho_j \in \Gamma(X, C^\infty(E))$. 
Then the equation 
\begin{align*}
G^{(N)}|_{V_j} &= g_j^{(N)}\rho_j+ \sum_{\ell\neq j} g_\ell^{(N)} \rho_\ell\\
&= g_j^{(N)} (1-\sum_{\ell \neq j} \rho_\ell) + \sum_{\ell\neq j} g_\ell^{(N)}\rho_\ell\\
&=g_j^{(N)} +\sum_{\ell \neq j}(g_\ell^{(N)}-g_j^{(N)})\rho_\ell= g_j^{(N)}+\sum_{\ell\neq j}f_{j\ell}^{(N)}\rho_\ell
\end{align*}
holds on $V_j$. 
Thus, we have 
\begin{align}\label{CDcorres1}
\delbar G^{(N)}=\sum_{\ell\neq j} f_{j\ell}^{(N)} \delbar \rho_\ell
\end{align}
on $V_j$. 
The right hand side of the equation $(\ref{CDcorres1})$ is a $(0,1)$-form whose class corresponds to $[\{(V_{j\ell}, f_{j\ell}^{(N)})\}]$ via the \v{C}ech--Dolbeault correspondence. 
The form $\delbar G^{(N)}$ converges compactly uniformly in all the derivatives of coefficients to $\textstyle \Omega:= \sum_{\ell\neq j}f_{j\ell} \delbar \rho_\ell=\sum_\ell f_{j\ell} \delbar \rho_\ell$ on $V_j$, which corresponds to $[\{(V_{j\ell}, f_{j\ell})\}]$. 
We have $[w]=[\Omega]\in H^1(X,E)$. 
Hence, there exists $\sigma\in\Gamma(X, C^\infty(E))$ such that $w=\Omega+\delbar \sigma$. 
We define $w_N:=\delbar (G^{(N)}+\sigma) \in B^{0,1}(E)$ for $N\in\Z_{\geq0}$. 
Then $w_N$ converges uniformly in all the derivatives of coefficients to $w$ on compact sets. 
Therefore, $\bar{B^{0,1}(E)}=Z^{0,1}(E)$. 
\end{proof}

\begin{remark}\label{cohom:equiv}
In the proof of Proposition $\ref{prop:dense}$, \v{C}ech--Dolbeault correspondence means that two vector spaces are equal as topological vector spaces (see \cite[Theorem $2.1$]{Lau}). 
According to \cite[Theorem $2.1$]{Lau}, the cohomology obtained by the space of currents with the topology of the uniform convergence on bounded sets (the strong dual topology) is also topologically isomorphic to the cohomology for a Leray covering or $C^\infty$ forms as in Proposition $\ref{prop:dense}$. 
\end{remark}


\subsection{Serre duality theorem}\label{serre}
In this subsection, we confirm that the vanishing of the topological space $(H^1(X,E))^\vee$ implies that $H^1_0(X,E)=0$ for $(X,E)$ as in Example \ref{eg:main} by using Serre duality theorem \cite{LT-sepa} and \cite{LT-serre} (see also \cite[Theorem $1.33$]{OhL}). 

\begin{claim}\label{serre:suffi}
For a pair $(X,E)$ as in Example \ref{eg:main}, the dual space $(H^1(X,E^{-1}))^\vee$ vanishes if and only if $H^1_0(X,E)=0$ holds.
\end{claim}

\begin{proof}
We regard the cohomology $H^1(X,E^{-1})$ as the quotient space $\Ker\ (\delbar: \mathscr{K}^{0,1}(E^{-1}) \to \mathscr{K}^{0,2}(E^{-1}))/ \delbar\mathscr{K}^{0,0}(E^{-1})$, where $\mathscr{K}^{p,q}(E^{-1})$ is the space of $E^{-1}$-valued $(p,q)$-currents (see Remark \ref{cohom:equiv}). 
Since $X$ is strongly $2$-complete (see \cite[p.417, Theorem $3.5$]{agbook}), we find that $H^{2,2}(X,E^{-1})=0$ by \cite[p.419, Proposition $4.1$]{agbook}. 
Moreover, we have $H^{0,0}_0 (X ,E)=0$ by identity theorem. 
By applying Serre duality theorem \cite[p.39]{OhL}, we find that $H^{0,1}_0(X,E)\cong(H^{2,1}(X,E^{-1}))^\vee$ (for their topologies, see Remark $\ref{top:cohom}$). 
Since the canonical bundle $K_X$ of $X$ is trivial,
\begin{align*}
H^1_0(X,E)&\cong(H^{2,1}(X,E^{-1}))^\vee \\
&=(H^{0,1}(X,K_X\otimes E^{-1}))^\vee =(H^1(X,E^{-1}))^\vee. 
\qedhere
\end{align*}
\end{proof}

\begin{remark}\label{top:cohom}
In the proof of Claim $\ref{serre:suffi}$, 
we topologize the space of currents and its subspaces by the topology of the uniform convergence on bounded sets, and the cohomology of currents by the quotient topology. 
On the other hand, the cohomology for $C^\infty$ forms with compact support has the quotient topology which comes from the topology of the uniform convergence of all the derivatives of coefficients with uniformly bounded supports for the space of $C^\infty$ forms with compact support.
\end{remark}

\subsection{Proof of Theorem \ref{thm:main}}\label{pf:main}
Let $X$ and $E$ as in Theorem $\ref{thm:main}$.
Take $T\in (H^1(X,E))^\vee$, namely,  $T: Z^{0,1}(E)/B^{0,1}(E)\to \C$ is a linear and continuous function. 
Denote by $\tilde{T}:Z^{0,1}(E) \to \C$ the lift of $T$. 
By definition of $\tilde{T}$, we have $\tilde{T}(B^{0,1}(E))=0$. 
Moreover,  $\tilde{T}\left(\bar{B^{0,1}(E)}\right)=0$ since $\tilde{T}$ is continuous. 
Now, $\bar{B^{0,1}(E)}=Z^{0,1}(E)$ holds by Proposition \ref{prop:dense}. 
Thus, $\tilde{T}(w)=0$ for all $w\in Z^{0,1}(E)$. 
Hence, we find that $T\equiv 0$, that is, $(H^1(X,E))^\vee=0$. 
Therefore we have $H_0^1(X,E)=0$ by Claim $\ref{serre:suffi}$. \qed

\begin{remark}
For a pair $(X,E)$ which satisfies the conditions $(i)$, $(ii)$, and $(iii)$ as in \S $\ref{notation:general}$, $(H^1(X,E))^\vee=0$. 
In fact, by a similar argument as in the proof of Theorem \ref{thm:main}, the vanishing of $(H^1(X,E))^\vee$ follows from the density of $B^{0,1}(E)$ . 
\end{remark}

\section{Another Proof of Theorem \ref{thm:main} when $p=\theta_1=0$}\label{section:proof}


\subsection{Notation and Settings}\label{notation}
In this subsection, we describe the notation and the settings which we will use in this section.

Let $X=X_{\tau, 0, q}$ and $E=E_{0, \theta_2}$ as in Example \ref{eg:main}. Namely, for $\tau \in \mathbb{H}$ and $q \in \R$, we define a complex Lie group $X$ by
\begin{align}
	X &= \C^2/\left\langle \begin{pmatrix} 0\\ 1\\ \end{pmatrix} , \begin{pmatrix} 1\\ 0\\ \end{pmatrix} , \begin{pmatrix} \tau\\ q\\ \end{pmatrix} \right\rangle. \nonumber
\end{align}
We denote by $(z,w)$ the coordinate of $\C^2$ and by $(\xi,\eta) $ the coordinate of $(\C^\ast)^2$ defined by $\xi = e^{2 \pi \sqrt{-1} z}$ and $ \eta=e^{2\pi \sqrt{-1} w}$. 
We set $\lam:= e^{2\pi \sqrt{-1} \tau}$ and $\mu:=e^{2\pi \sqrt{-1} q}$. 
Define $\sigma \in \rm{Aut}((\C^\ast)^2)$ by letting $ \sigma(\xi , \eta) := (\lam \xi , \mu \eta)$. 
Let $p \colon (\C^\ast)^2 \to (\C^\ast)^2/\Z$ be the quotient map induced by $\sigma$. 
We can regard the quotient $(\C^\ast)^2/\Z$ as $X$.
Denote by $C$ the elliptic curve $\C/\left\langle 1, \tau \right\rangle$. 
Let $\pi \colon X \to C$ be the mapping induced by the first projection $\C^2 \to \C \ ; (z,w)\mapsto z$. 
Note that $\pi \colon X\to C$ can be regarded as a $\C^\ast$-fiber bundle structure on $C$. 
Let $E_C$ be the flat line bundle which is constructed from the unitary representation
\begin{align*}
	\rho \colon \left\langle 1,\tau \right\rangle \to {\rm U}(1)\hsp ; 1 \mapsto 1 \ \text{and}\  \tau \mapsto \nu := e^{2\pi \i \theta_2} \hsp ( \, \theta_2 \in \R).
\end{align*}
Note that $E$ can be regarded as the flat line bundle $\pi^\ast E_C$ on $X$. 
Moreover, let $F_C$ be the flat line bundle which is constructed from the unitary representation
\begin{align*}
	\rho' \colon \left\langle 1,\tau \right\rangle \to {\rm U}(1)\hsp ; 1 \mapsto 1 \ \text{and}\ \tau \mapsto \mu .
\end{align*}
For $n\in\Z$, $D_n$ denotes the set $\{(\xi,\eta)\in(\C^\ast)^2 \mid |\lam|^{n+1}\leq |\xi| \leq |\lam|^n \}$. Note that $D_n=\sigma^n(D_0)$ holds.


\subsection{Observation}\label{obs}
We use H\"{o}rmander estimate for the proof in \S\ref{section:proof} since it is difficult to confirm the convergence of $\hat{g}$ on $(\C^\ast)^2$ such that $\delbar\hat{g}=\tilde{f}$ and $\sigma^\ast\hat{g}=\nu\hat{g}$. 
In this subsection, we attempt to construct such a section $\hat{g}$ without H\"{o}rmander estimate and observe the difficulty.

Let $f\in\Gamma(X,\mathscr{A}^{0,1}_{C^\infty}(E))$ be a $(0,1)$-form with compact support such that $\delbar f=0$, where $\mathscr{A}^{p,q}_{C^\infty}(E)$ is the sheaf of $E$-valued $(p,q)$-forms of class $C^\infty$ on $X$. Then the $(0,1)$-form $\tilde{f}:=p^\ast f$ on $(\C^\ast)^2$ satisfies the periodic condition $\sigma^\ast\tilde{f}=\nu\tilde{f}$.

By Steinness of $(\C^\ast)^2$ and the coherence of $\mathcal{O}_{(\C^\ast)^2}$, we obtain $H^1((\C^\ast)^2,\mathcal{O}_{(\C^\ast)^2})=0$. 
Hence, there exists $\hat{g}\in \Gamma((\C^\ast)^2, \mathscr{A}_{C^\infty}^{0,0}(E))$ such that $\delbar \hat{g}=\tilde{f}$. 
This argument does not depend on a metric or a weight function. 
If $\hat{g}$ satisfies the periodic condition $\sigma^\ast\hat{g}=\nu\hat{g}$, we can take a section $g\in \Gamma(X,\mathscr{A}_{C^\infty}^{0,0}(E))$ such that $\delbar g=f$. 
In the following, we attempt to construct a new solution $\tilde{g}\in\Gamma((\C^\ast)^2,\mathscr{A}_{C^\infty}^{0,0}(E))$ of the $\delbar$-equation which satisfies the periodic condition by modifying $\hat{g}$.

Let $\hat{g}$ be as above. Then, $F:= \nu^{-1}\sigma^\ast \hat{g} - \hat{g}$ is holomorphic on $(\C^\ast)^2$ since the equations
\begin{align*}
\delbar F= \delbar(\nu^{-1}\sigma^\ast \hat{g}-\hat{g}\,) =\nu^{-1}\sigma^\ast(\delbar\hat{g}\,)-\tilde{f}=\nu^{-1}\sigma^\ast\tilde{f}-\tilde{f}=0.
\end{align*}
hold.

We want to construct a holomorphic function $G \colon (\C^\ast)^2 \to \C$ which satisfies $\sigma^\ast G-\nu G \equiv \nu F$, since then $\tilde{g}:=\hat{g}-G$ satisfies 
\begin{align*}
\left\{ \begin{array}{ll} \delbar\tilde{g}=\tilde{f}, \\ \sigma^\ast\tilde{g}=\nu\tilde{g}. \end{array} \right.
\end{align*}

We attempt to construct such a function $G$. 
For $\xi \in \C^\ast$, the holomorphic function $F(\xi,-)\colon \C^\ast \to \C$ has  Laurent expansion 
\begin{align*}
F(\xi,\eta)=\sum_{n=-\infty}^\infty a_n(\xi)\eta^n,
\end{align*}
where $\textstyle a_n(\xi):=\frac{1}{2\pi \i}\int_\gamma \frac{F(\xi,\eta)}{\eta^{n+1}}d\eta$ and $\gamma$ is a smooth Jordan curve around $\eta=0$. Now $a_n$ is holomorphic in $\xi$. Moreover, each $a_n$ has Laurent expansion
\begin{align*}
a_n(\xi)=\sum_{m=-\infty}^\infty a_{n,m}\xi^m.
\end{align*}
Let us define a formal power series $G(\xi,\eta):=\sum_{n,m=-\infty}^\infty\frac{\nu}{\lam^m\mu^n-\nu}\cdot a_{n,m}\xi^m\eta^n$. 
Then
\begin{align*}
\sigma^\ast G(\xi,\eta)-\nu G(\xi,\eta)&=G(\lam\xi,\mu\eta)-\nu G(\xi,\eta)\\
&=\sum_{n,m=-\infty}^\infty \left(\frac{\lam^m\mu^n\nu}{\lam^m\mu^n-\nu}-\frac{\nu^2}{\lam^m\mu^n-\nu}\right)\cdot a_{n,m}\xi^m\eta^n\\
&=\nu F(\xi,\eta)
\end{align*}
formally holds.
Therefore, we obtain
\begin{align*}
\sigma^\ast G-\nu G=\nu F = \sigma^\ast\hat{g}-\nu\hat{g}
\end{align*}
holds formally.

It suffices to prove the convergence of the formal power series $G$. However, it is difficult to prove the convergence since the estimates of the coefficients depend on number theoretical conditions for $\tau, q$, and $\theta_2$. In this section, we construct $\hat{g}$ by using an appropriate metric and a weight function to avoid this difficulty. 


\subsection{Proof of Theorem \ref{thm:main}}\label{pf:sub}

In this subsection, we prove Theorem \ref{thm:main}. 
We use the notation \S $\ref{notation}$ and \S $\ref{obs}$.

First, let us fix a weight function and a metric. 
We define a weight function as
\begin{align*}
\psi(\xi,\eta):= (\log |\xi|^2)^2+\log\left(|\eta|^2+\frac{1}{|\eta|^2}\right).
\end{align*}
Then $\psi$ is a strongly plurisubharmonic function on $(\C^\ast)^2$. 
Note that 
\begin{align}\label{ddbarpsi}
\sqrt{-1}\ddbar\psi &= \sqrt{-1} \frac{2}{|\xi|^2}d\xi \wedge d\bar{\xi} + \sqrt{-1}\frac{4}{|\eta|^2 \left(|\eta|^2+\frac{1}{|\eta|^2}\right)^2} d\eta\wedge d\bar{\eta}.
\end{align}
We define a volume form $d\tilde{\lam}$ by letting
\begin{align*}
d\tilde{\lam}:=\sqrt{-1}\ddbar\left(\frac{(\log |\xi|^2)^2}{2}\right)\wedge\ddbar\left( |\eta|^2+\frac{1}{|\eta|^2} \right),
\end{align*}
and a K\"{a}hler form $\omega$ by letting
\begin{align*}
\omega&:=\sqrt{-1}\ddbar\left(\frac{(\log |\xi|^2)^2}{2}\right)+\sqrt{-1}\ddbar\left(|\eta|^2+\frac{1}{|\eta|^2}\right)\\
&= \sqrt{-1} \frac{1}{|\xi|^2} d\xi \wedge d\bar{\xi} + \sqrt{-1}\left( 1+\frac{1}{|\eta|^4} \right) d\eta\wedge d\bar{\eta}.
\end{align*} 
Then, $d\tlam$ is a volume form of $\omega$. 
For $\psi$ and $\omega$ defined as above, we show the following claim. 

\begin{claim}\label{claim1}
\begin{align*}
\int_{(\C^\ast)^2} \left| \tilde{f}\, \right|^2_\omega e^{-\psi} d\tlam <\infty.
\end{align*}
\end{claim}

\begin{proof}
For $n\in\Z$, each domain $D_n$ satisfies $D_n=\sigma^n(D_0)$. 
Thus, we obtain the equation
\begin{align}
\int_{(\C^\ast)^2} \left|\tilde{f}\, \right|_\omega^2 e^{-\psi} d\tlam&=\sum_{n=-\infty}^\infty \int_{D_n} \left| \tilde{f}\, \right|^2_\omega e^{-\psi}d\tlam \nonumber\\
&=\sum_{n=-\infty}^\infty \int_{D_0} (\sigma^n)^\ast \left( \left| \tilde{f}\, \right|^2_\omega e^{-\psi} d\tlam \right).\label{sumeq}
\end{align}
Since the K\"{a}hler form $\omega$ is $\sigma$-invariant, the measure $d\tlam$ is also $\sigma$-invariant. 
Thus, we have
\begin{align*}
\sigma^\ast \left| \tilde{f}\, \right|^2_\omega= \left| \sigma^\ast\tilde{f}\, \right|^2_{\sigma^\ast \omega}= \left| \nu\tilde{f}\, \right|^2_\omega= \left| \tilde{f}\, \right|^2_\omega.
\end{align*}
By the equation $(\ref{sumeq})$ and definition of $\psi$, 
\begin{align*}
\sum_{n=-\infty}^\infty \int_{D_0} (\sigma^n)^\ast \left( \left| \tilde{f}\, \right|^2_\omega e^{-\psi} d\tlam \right) &=\sum_{n=-\infty}^\infty \int_{D_0} \left|\tilde{f}\, \right|^2_\omega e^{-\psi(\lam^n \xi, \mu^n \eta)}d\tlam\\
&=\sum_{n=-\infty}^\infty \int_{D_0} \left|\tilde{f}\, \right|^2_\omega e^{-(\log|\xi|^2+n\log|\lam|^2)^2} \frac{1}{|\eta|^2+\frac{1}{|\eta|^2}} d\tlam
.\end{align*}
Now, the inequality
\begin{align*}
(\log|\xi|^2+n\log|\lam|^2)^2\geq \left\{ \begin{array}{ll} ((n+1)\log|\lam|^2)^2 & (n< 0)\\ (n\log|\lam|^2)^2 & (n\geq 0)\\ \end{array} \right.
\end{align*}
holds on $D_0$. 
Therefore, 
\begin{align}
&\sum_{n=-\infty}^\infty \int_{D_0} \left|\tilde{f}\, \right|^2_\omega e^{-(\log|\xi|^2+n\log|\lam|^2)^2} \frac{1}{|\eta|^2+\frac{1}{|\eta|^2}} d\tlam \\
&\leq \sum_{n\geq 0} \int_{D_0}\left| \tilde{f}\, \right|^2_\omega e^{-n^2(\log|\lam|^2)^2} \frac{1}{|\eta|^2+\frac{1}{|\eta|^2}} d\tlam 
+\sum_{n<0} \int_{D_0} \left| \tilde{f}\, \right|^2_\omega e^{-(n+1)^2(\log|\lam|^2)^2} \frac{1}{|\eta|^2+\frac{1}{|\eta|^2}} d\tlam \nonumber \\
&\leq 2 \sum_{n\geq 0} \left(e^{(\log|\lam|^2)^2}\right)^{-|n|} \int_{D_0} \left| \tilde{f}\, \right|^2_\omega  \frac{1}{|\eta|^2+\frac{1}{|\eta|^2}} d\tlam.\nonumber \\
&\leq \sum_{n\geq 0} \left(e^{(\log|\lam|^2)^2}\right)^{-|n|} \int_{D_0} \left| \tilde{f}\, \right|^2_\omega d\tlam. \label{eqn311}
\end{align}
Let $M$ be a positive number such that $\textstyle \int_{D_0} \left| \tilde{f}\, \right|^2_\omega d\tlam <M$. By the equation $(\ref{eqn311})$, we have
\begin{align*}
 \sum_{n\geq 0} \int_{D_0} \left| \tilde{f}\, \right|^2_\omega \left(e^{(\log|\lam|^2)^2} \right)^{-|n|} d\tlam &<   \sum_{n\geq 0} M \left(e^{(\log|\lam|^2)^2}\right)^{-|n|}\\
&<\infty. \qedhere
\end{align*}
\end{proof}

We take a sufficiently large constant $K$ such that $\Supp \tilde{f} \subset \{ (\xi,\eta)\in (\C^\ast)^2 \mid |\log|\eta||<K \}.$ 
Now,
\begin{align*}
1+\frac{1}{|\eta|^4},\hsp \frac{1}{|\eta|^2(|\eta|^2+\frac{1}{|\eta|^2})^2}
\end{align*}
are bounded on $\{ |\log|\eta||<K \}$. 
Hence, the eigenvalue of the complex Hessian of $\i\ddbar\psi$ with $\omega$ is identically to $2$ in the direction $\del/\del\xi$, and is positive and bounded form above in the direction $\del/\del\eta$. 
Then there exists a constant $\delta>0$ such that $\delta$ is smaller than the minimal eigenvalue of $\psi$. 
We have  the following lemma.

\begin{lemma}\label{lemmaHor}
\begin{align*}
\int_{(\C^\ast)^2} \left|\tilde{f}\,\wedge \frac{d\xi}{\xi}\wedge\frac{d\eta}{\eta}\right|^2_\omega e^{-\psi} d\tlam <\infty.
\end{align*}
\end{lemma}

\begin{proof}
By definition of $\omega$,
\begin{align*}
\left| \frac{d\xi}{\xi} \right|^2_\omega =1 ,\quad
\left| \frac{d\eta}{\eta} \right|^2_\omega = \frac{1}{|\eta|^2+\frac{1}{|\eta|^2}} \leq \frac{1}{2}.
\end{align*}
We obtain the following inequality from Lemma \ref{claim1}.
\begin{align*}
\int_{(\C^\ast)^2} \left|\tilde{f}\,\wedge \frac{d\xi}{\xi}\wedge\frac{d\eta}{\eta}\right|^2_\omega e^{-\psi} d\tlam
\leq \frac{1}{2} \int_{(\C^\ast)^2} \left|\tilde{f}\right|^2 e^{-\psi} d\tlam < \infty.
\end{align*}
\end{proof}

By Lemma \ref{lemmaHor} and the existence of such $\delta$, we can apply H\"{o}rmander estimate \cite[Corollary 5.3]{Analy} to 
\begin{align*}
\tilde{f}\,\wedge\frac{d\xi}{\xi}\wedge \frac{d\eta}{\eta}\in\Gamma((\C^\ast)^2,\mathscr{A}_{C^\infty}^{2,1}(E))
\end{align*}
to conclude that there exists
\begin{align*}
\hat{g}\,\frac{d\xi}{\xi}\wedge \frac{d\eta}{\eta} \in\Gamma((\C^\ast)^2,\mathscr{A}_{C^\infty}^{2,0}(E))
\end{align*}
such that $\delbar\,{\hat{g}}=\tilde{f}$ and 
\begin{align}
\int_{(\C^\ast)^2}\left|\hat{g}\,\frac{d\xi}{\xi}\right.&\left.\wedge \frac{d\eta}{\eta}\right|^2_\omega e^{-\psi} d\tlam<\infty.\label{horeq}
\end{align}
By the inequality $(\ref{horeq})$ and Fubini's theorem, there exists a set $U\subset \C^\ast$ with full measure such that 
\begin{align}
\int_{\eta\in\C^\ast} \left| \hat{g}(\xi,\eta) \right|^2 \frac{\sqrt{-1}}{|\eta|^2 \left(|\eta|^2+\frac{1}{|\eta|^2}\right)} d\eta\wedge d\bar{\eta}<\infty \label{gtildeine}
\end{align}
for any $\xi\in U$. 
For $\hat{g}$ we obtained above, we define $F:= \nu^{-1}(\sigma^\ast \hat{g}\,)-\hat{g}$. Since $\delbar \hat{g}=\tilde{f}$ and $\tilde{f}$ satisfies the periodic condition $\sigma^\ast\tilde{f}=\nu\tilde{f}$, we have $\delbar F=0$. 
By the inequality $(\ref{gtildeine})$, we find that 
\begin{align*}
\int_{\eta\in\C^\ast} |F(\xi,\eta)|^2 \frac{\sqrt{-1}}{|\eta|^2 \left(|\eta|^2+\frac{1}{|\eta|^2}\right)} d\eta&\wedge d\bar{\eta}
\leq \int_{\eta\in\C^\ast}  |\nu^{-1}(\sigma^\ast \hat{g}\,)-\hat{g}\,|^2 \frac{\sqrt{-1}}{|\eta|^2 \left(|\eta|^2+\frac{1}{|\eta|^2}\right)} d\eta\wedge d\bar{\eta}\\
&\leq \int_{\eta\in\C^\ast}  \left( |\nu^{-1}(\sigma^\ast \hat{g}\,)|^2+|\hat{g}\,|^2 \right) \frac{\sqrt{-1}}{|\eta|^2 \left(|\eta|^2+\frac{1}{|\eta|^2}\right)} d\eta\wedge d\bar{\eta}\\
\end{align*}
holds for any $\xi\in U$. 
By the equality $|\nu^{-1}(\sigma^\ast \hat{g}\,(\xi,\eta))|=|\hat{g}\,(\lam\xi, \mu\eta)|$ and the inequality $(\ref{gtildeine})$, we find that
\begin{align*}
&\int_{\eta\in\C^\ast} |\nu^{-1}(\sigma^\ast \hat{g}\,)(\xi,\eta)|^2 \frac{\sqrt{-1}}{|\eta|^2 \left(|\eta|^2+\frac{1}{|\eta|^2}\right)} d\eta\wedge d\bar{\eta}\\
&= \int_{\eta\in\C^\ast} |\hat{g}\,(\lam\xi,\mu\eta)|^2 \frac{\sqrt{-1}}{|\eta|^2 \left(|\eta|^2+\frac{1}{|\eta|^2}\right)} d\eta\wedge d\bar{\eta}\\
&< \infty
\end{align*}
for any $\lam\xi \in U$. 
Since the set $U \cap \sigma^\ast (U)$ has full measure and
\begin{align*}
&\int_{\eta\in\C^\ast}  \left( |\nu^{-1}(\sigma^\ast \hat{g}\,)|^2+|\hat{g}\,|^2 \right) \frac{\sqrt{-1}}{|\eta|^2 \left(|\eta|^2+\frac{1}{|\eta|^2}\right)} d\eta\wedge d\bar{\eta}\\
&\leq \int_{\eta\in\C^\ast}  \left( |\hat{g}\,(\lam\xi, \mu\eta)|^2+|\hat{g}\,(\xi,\eta)|^2 \right) \frac{\sqrt{-1}}{|\eta|^2 \left(|\eta|^2+\frac{1}{|\eta|^2}\right)} d\eta\wedge d\bar{\eta}\\
&< \infty
\end{align*}
holds for any $\xi \in U \cap \sigma^\ast (U)$, we obtain the inequality
\begin{align*}
\int_{\eta\in\C^\ast} |F(\xi,\eta)|^2 \frac{\sqrt{-1}}{|\eta|^2 \left(|\eta|^2+\frac{1}{|\eta|^2}\right)} d\eta\wedge d\bar{\eta} < \infty
\end{align*}
for almost every $\xi\in\C^\ast$.


\begin{lemma}\label{Riemannext} 
Let $h(\eta)$ be a holomorphic function on $\C^\ast$. If the inequality
 \begin{align*}
\int_{\eta\in\C^\ast} |h(\eta)|^2 \frac{\sqrt{-1}}{|\eta|^2 \left(|\eta|^2+\frac{1}{|\eta|^2}\right)} d\eta\wedge d\bar{\eta}<\infty
\end{align*}
holds, then $h$ is a constant function.
\end{lemma}

\begin{proof}
On a sufficiently small neighborhood $U_0$ of $\eta=0$, we have
\begin{align*}
\infty&>\int_{\eta\in\C^\ast} |h(\eta)|^2 \frac{\sqrt{-1}}{|\eta|^2 \left(|\eta|^2+\frac{1}{|\eta|^2}\right)} d\eta\wedge d\bar{\eta}\\
&>\int_{U_0\setminus \{0\}} |h(\eta)|^2 \frac{1}{|\eta|^2}\cdot\frac{\sqrt{-1}}{2{|\eta|^{-2}}}d\eta\wedge d\bar{\eta}\\
&=\int_{U_0\setminus \{0\}} |h(\eta)|^2 \frac{\sqrt{-1}}{2} d\eta\wedge d\bar{\eta}.
\end{align*}
We find that $h(\eta)$ is a holomorphic $L^2$ function with Lebesgue measure around $\eta=0$. By Riemann's extension theorem, $h(\eta)$ induces a holomorphic function extended holomorphically to $\eta=0$. In the same manner, on a sufficiently small neighborhood $U_\infty$ of $\eta=\infty$, we have
\begin{align*}
\infty&>\int_{\hat{\eta}\in\C^\ast} \left|h\left(\frac{1}{\hat{\eta}}\right)\right|^2 |\hat{\eta}|^2 \frac{\sqrt{-1}}{|\hat{\eta}|^2+\frac{1}{|\hat{\eta}|^2}} \frac{d\eta\wedge d\bar{\eta}}{|\hat{\eta}|^4}\\
&>\int_{U_\infty\setminus\{ \hat{\eta}=0\}} \left|h\left(\frac{1}{\hat{\eta}}\right) \right|^2 |\hat{\eta}|^2\cdot\frac{\sqrt{-1}}{2|\hat{\eta}|^{-2}}  \frac{d\eta\wedge d\bar{\eta}}{|\hat{\eta}|^4}\\
&=\int_{U_\infty\setminus\{ \hat{\eta}=0 \}} \left|h\left(\frac{1}{\hat{\eta}}\right) \right|^2 \frac{\sqrt{-1}}{2}d\hat{\eta}\wedge d\bar{\hat{\eta}},
\end{align*}
where $\hat{\eta}:=1/\eta$. 
Then we obtain a holomorphic function extended to $\eta=\infty$. 
Therefore, $h(\eta)$ is  holomorphic on $\mathbb{P}^1$. We find that $h(\eta)$ is a constant function.
\end{proof}

By Lemma \ref{Riemannext}, we find that $F(\xi,-)$ is a constant function on $\mathbb{P}^1$ for almost every $\xi\in\C^\ast$. 
Hence, for Laurent series expansion
\begin{align*}
F(\xi,\eta)=\sum^\infty_{n=-\infty}a_n(\xi)\eta^n
\end{align*}
at almost every $\xi\in\C^\ast$, we obtain $a_n(\xi)\equiv 0$ with $n\neq0$. 
Since each $a_n(\xi)$ is holomorphic, we find that $a_n(\xi)\equiv0 $ for any $\xi\in\C^\ast$ and any $n\in\Z$ with $n\neq0$. 
Therefore $F(\xi,\eta)=a_0(\xi)$.

For Laurent expansion
\begin{align*}
a_0(\xi)=\sum^\infty_{m=-\infty}a_{0,m}\xi^m,
\end{align*}
we define $\textstyle A(\xi):=\sum^\infty_{m=-\infty}\frac{1}{\lam^m\nu^{-1}-1}\cdot a_{0,m}\xi^m$. Then, the following lemma holds. 

\begin{lemma}
The power series $A(\xi)$ is holomorphic on $\C^\ast$.
\end{lemma}

\begin{proof}
By definition of $A(\xi)$, 
\begin{align}
A(\xi)&=\sum^\infty_{m=-\infty}\frac{1}{\lam^m\nu^{-1}-1}\cdot a_{0,m}\xi^m \nonumber\\
&=\sum_{m=-\infty}^0 \frac{1}{\lam^m\nu^{-1}-1}\cdot a_{0,m}\xi^m +\sum_{m=1}^\infty \frac{1}{\lam^m\nu^{-1}-1}\cdot a_{0,m}\xi^m. \label{fml311}
\end{align}
In the power series $(\ref{fml311})$, the first part $\textstyle \sum_{m=-\infty}^0 \frac{1}{\lam^m\nu^{-1}-1}\cdot a_{0,m}\xi^m$ is compactly absolutely convergent by $|\lam|<1$. 
We focus on the convergence of the second part $\textstyle \sum_{m=1}^\infty \frac{1}{\lam^m\nu^{-1}-1}\cdot a_{0,m}\xi^m$. 
We set $A_m:=\frac{a_{0,m}}{\lam^m\nu^{-1}-1}$. 
By the equation
\begin{align*}
\frac{A_{m}}{A_{m+1}}=\frac{a_{0,m+1}}{a_{0,m}}\cdot\frac{\lam^{m}\nu^{-1}-1}{\lam^{m+1}\nu^{-1}-1}
\end{align*}
and definition of $a_{0,m}$, we find that $\lim_{n\to\infty} |A_{m}/A_{m+1}|=\infty$. 
Thus, the second part is also compactly absolutely convergent. 
Therefore, $A(\xi)$ is compactly absolutely convergent.
\end{proof}
For each $(\xi,\eta)\in (\C^\ast)^2$, we have
\begin{align*}
\nu^{-1}(\sigma^\ast A(\xi))-A(\xi)&=\nu^{-1}(A(\lam\xi))-A(\xi)\\
&=\nu^{-1}\left(\sum\frac{1}{\lam^m\nu^{-1}-1}a_{0,m}(\lam\xi)^m\right)-\sum\frac{1}{\lam^m\nu^{-1}-1}a_{0,m}\xi^m\\
&=\sum a_{0,m}\xi^m=a_0(\xi).
\end{align*}
By letting $\tilde{g}:=\hat{g}-A$, $\sigma^\ast\tilde{g}=\nu\tilde{g}$ and $\delbar\tilde{g}=\tilde{f}$ hold. Therefore we obtain $g\in \Gamma(X,\mathscr{A}^{0,0}_{C^\infty}(E))$ such that $\tilde{g}=p^\ast g$ and $\delbar g=f$.

Finally, we prove that the section $g$ has compact support. 
It is sufficient to prove that $g\equiv0$ on $M_K:=\{ |\log|\eta||>K \} \subset X$. 
Moreover, by definition of $M_K$, it suffices to show $H^0(M_K, E)=0$ since $g|_{M_K}$ is a holomorphic section. 
The complex surface $X$ can be identified with the complement of the zero-section in the total space of $F_C$. 
Denote by $\tilde{X}$ what can be constructed by adding the infinity-section to $X$. 
Then $\tilde{X}$ admits a $\mathbb{P}^1$-bundle structure over $C$. 
Let $W^+$ be the infinity-section as a submanifold of $\widehat{M_K^+} :=\{ \log|\eta|>K\} \subset \tilde{X}$ and $W^-$ be the zero-section as a submanifold of $\widehat{M_K^-} :=\{ \log|\eta|<-K\} \subset \tilde{X}$. 
We apply $Z=\widehat{M_K^\pm}$ and $W=W^\pm$ to Lemma \ref{lem:vanishing_of_H0_on_nbhd_of_C}. 
Note that $\textstyle N_{W^\pm/\widehat{M_K^\pm}}\cong F_C^{\mp 1}$. 

\begin{lemma}\label{lem:vanishing_of_H0_on_nbhd_of_C}
Let $E$ be a holomorphic line bundle on a connected complex manifold $Z$ and $W\subset Z$ be a connected compact non-singular hypersurface. 
If $E|_W\otimes N_{W/Z}^{-n}$ on W is a non-trivial holomorphic flat line bundle for all $n\in\Z_{\geq0}$, then $H^0(Z, E)=0$. 
\end{lemma}

\begin{proof}
Let $[g]\in H^0(Z, E)$. 
We prove that $g\equiv 0$ by contradiction. 
If $g\not\equiv 0$, the multiplicity $m:={\rm mult}_W\,g$ of zeros of $g$ on $W$ satisfies $0\leq m<\infty$. 
Since we regard $g$ as a section of $\mathcal{O}_Z(E)\otimes I_W^m$ for $m$, we obtain
\begin{align*}
g\ {\rm mod}\ I_W^{m+1}\ \in\ H^0(Z, \mathcal{O}_Z(E)\otimes I_W^m/I_W^{m+1}),
\end{align*}
where $I_W\subset \mathcal{O}_Z$ is the defining ideal sheaf of $W$. 
A sheaf $\mathcal{O}_Z(E)\otimes I_W^m/I_W^{m+1}$ has support only on $W$ by restricting it to $W$, and coincides with $\mathcal{O}_W(E|_W\otimes N_{W/Z}^{-m})$. Thus, we get
\begin{align*}
g\ {\rm mod}\ I_W^{m+1}\ \in\ H^0(W, \mathcal{O}_W(E|_W\otimes N_{W/Z}^{-m})).
\end{align*}
We apply $n=m$ to the assumption. 
Since $E|_W\otimes N_{W/Z}^{-m}$ is a non-trivial unitary flat line bundle, the holomorphic section is identically zero (See \cite[Lemma 2.3]{HaKo} for example). 
Then $H^0(W, \mathcal{O}_W(E|_W\otimes N_{W/Z}^{-m}))=0$. 
Therefore we find that $g$ vanishes at least $m+1$ multiples along $W$:
\begin{align*}
g\ {\rm mod}\ I_W^{m+1}\equiv 0.
\end{align*}
This contradicts definition of the integer $m$.
\end{proof}

Since $A(\xi)$ does not depend on $\eta$ and the inequality $(\ref{gtildeine})$ holds, on $\{ (\xi,\eta) \in (\C^\ast)^2 \mid |\log|\eta||>K \}$, $\tilde{g}(\xi,\eta)$ induces a holomorphic function extended holomorphically to $\eta=0$ and $\infty$ by using Riemann's extension theorem for $\tilde{g}$.
Thus, $g$ is defined on $\widehat{M_K^\pm}$. 
Hence, we find that $g|_{\widehat{M_K^\pm}} \in H^0(\widehat{M_K^\pm}, E)$ since $g|_{M_K}$ is a holomorphic section. 
By applying Lemma $\ref{lem:vanishing_of_H0_on_nbhd_of_C}$ with $Z=\widehat{M_K^\pm}$ and $W=W^\pm$, we have $H^0(\widehat{M_K^\pm}, E)=0$. 
Therefore, we obtain that $g|_{\widehat{M_K^\pm}} \equiv0$. 
\qed






\begin{thebibliography}{99}
\bibitem{AK} \textrm{Y. Abe, K. Kopfermann}: \rm Toroidal groups. Line bundles, cohomology and quasi-abelian varieties, Lecture Notes in Mathematics, {\bf 1759}, Springer-Verlag, Berlin, 2001.\rm
\bibitem{Abetoro} \textrm{Y. Abe}: \it Cohomology groups of sections of homogeneous line bundles over a toroidal group, \rm Kyushu. J. Math. {\bf70} (2016), 149--166.
\bibitem{Analy} \textrm{Demailly, J.-P}: \rm Analytic Methods in Algebraic Geometry. \rm Higher Education Press, Surveys of Modern Mathematics, Vol.1. 2010.
\bibitem{agbook} \textrm{Demailly, J.-P}: \rm Complex analytic and differential geometry. \rm https://www-fourier.ujf-grenoble.fr/~demailly/manuscripts/agbook.pdf. Accessed 7 November 2024.
\bibitem{HaKo} \textrm{Y. Hashimoto, T. Koike}: \it Ueda's lemma via uniform H\"{o}rmander estimates for flat line bundles, \rm to appear in Kyoto J. Math.
\bibitem{Lau} \textrm{H. B. Laufer}: \it On Serre duality and envelopes of holomorphy, \rm Trans. Amer. Math. Soc. {\bf128} (1967) 414--436. 
\bibitem{LT-sepa} \textrm{C. Larent-Thi\'{e}baut, J. Leiterer}: \it A separation theorem and Serre duality for the Dolbeault cohomology, \rm Ark. Mat., {\bf40} (2002), 301--321.
\bibitem{LT-serre} \textrm{C. Larent-Thi\'{e}baut, J. Leiterer}: \it On Serre duality, \rm Bull. Sci. math. {\bf124}, 2 (2000) 93--106.
\bibitem{Kazama} \textrm{H. Kazama}: \it $\delbar$-cohomology of (H,C)-groups, \rm Publ. Res. Inst. Math. Sci. {\bf20} (1984), 297--317.
\bibitem{KaTa} \textrm{H. Kazama, S. Takayama}: \it $\ddbar$-problem on weakly $1$-complete K\"{a}hler manifolds, \rm Nagoya Math. J. {\bf155} (1999), 81--94.
\bibitem{Kodel} \textrm{T. Koike}: \it $\delbar$ cohomology of the complement of a semi-positive anticanonical divisor of a compact surface, \rm Math. Z. 308, {\bf22} (2024).
\bibitem{OhL} \textrm{T. Ohsawa}: \rm $L^2$ Approaches in Several Complex Variables, \rm Springer Japan KK, 2018.
\bibitem{BMker} \textrm{J. Tanaka}: \it Studies on generalizations of the Bochner--Martinelli kernel to complex abelian Lie groups, \rm master thesis (Japanese).
\end{thebibliography}
\end{document}